\documentclass[10pt,a4paper]{amsart}
\usepackage{amsmath,amssymb,color,setspace}
\usepackage{hyperref}
\usepackage[utf8,utf8x]{inputenc}
\usepackage{fullpage}
\usepackage{enumitem,xspace}
\usepackage{longtable}

\newtheorem{Lemma}{Lemma}[section]
\newtheorem{Theorem}[Lemma]{Theorem}
\newtheorem{Proposition}[Lemma]{Proposition}

\theoremstyle{definition}
\newtheorem{Definition}[Lemma]{Definition}

\theoremstyle{remark}
\newtheorem{Remark}[Lemma]{Remark}

\numberwithin{equation}{section}

\newcommand{\df}[1]{\emph{#1}}
\newcommand{\RS}{\mathcal{R}}
\DeclareMathOperator{\Aut}{Aut}
\newcommand{\coord}{\Upsilon}

\newcommand{\GRS}{generalized root system\xspace}
\newcommand{\GRSs}{generalized root systems\xspace}
\newcommand{\WG}{\mathcal{W}}
\newcommand{\redGRS}{reduced root set\xspace}
\newcommand{\redGRSs}{reduced root sets\xspace}

\title[A classification of generalized root systems]
{A classification of generalized root systems}

\author{M.~Cuntz}
\address{Michael Cuntz,
Institut f\"ur Algebra, Zahlentheorie und Diskrete Mathematik,
Fakult\"at f\"ur Mathematik und Physik,
Leibniz Universit\"at Hannover,
Welfengarten 1,
D-30167 Hannover, Germany}
\email{cuntz@math.uni-hannover.de}

\author{B.~M\"uhlherr}
\address{Bernhard M\"uhlherr,
Mathematisches Institut, Arndtstra{\ss}e 2, 35392 Gie{\ss}en, Germany}
\email{bernhard.muehlherr@math.uni-giessen.de}


\begin{document}

\begin{abstract}
Dimitrov and Fioresi introduced an object that they call a generalized root system. This is a finite set of vectors in a euclidean space satisfying certain compatibilities between angles and sums and differences of elements. They conjecture that every generalized root system is equivalent to one associated to a restriction of a Weyl arrangement. In this note we prove the conjecture and provide a complete classification of generalized root systems up to equivalence.
\end{abstract}

\maketitle

\section{Introduction}

A \emph{\GRS} as introduced by Dimitrov and Fioresi \cite{DF23} is a finite set of vectors in a euclidean vector space which is closed under sums or differences of elements depending on the angle between the vectors (see Definition \ref{def:GRS}). Although the definition is very short, it generalizes the concept of a classic root system (i.e.\ a root system in the sense of \cite[\S 9.2]{Hu72}) in a very nice way.

At first sight, the definition of a \GRS looks much more general than the definition of a classic root system because it does not require any symmetries. However, 
in Sections 1 and 2 of \cite{DF23} it is observed that \GRS share many
fundamental properties of classic root systems. A central notion in the context of \GRS is that of a \emph{quotient} (cf.\ Def.\ \ref{def:quotient}).
The key observation is that a quotient of a \GRS is again a \GRS. In particular, the quotient of a classic root system is a \GRS which is not classic in the generic case. In fact, one may interpret the quotients of classic root systems as combinatorial generalizations of the relative root systems known form the theory of algebraic groups. 
In the language of arrangements of hyperplanes, a quotient of a root system is a restriction of a Weyl arrangement to an element of its intersection lattice.

In Section 5 of \cite{DF23} the irreducible \GRSs of rank 2 are classified and it turns out that each such \GRS is equivalent to a rank 2 quotient of a classic root system.
In view of this, it is conjectured in \cite{DF23} that each \GRS is equivalent to a quotient of a classic root system (Conjecture 5.8 in \cite{DF23}). In this note we prove this conjecture for all \GRSs that do not have an irreducible direct factor of rank 1. Here is our main result (see Subsection \ref{proof:main} for a proof).

\begin{Theorem} \label{mainresult}
Each irreducible \GRS of rank at least 2 is equivalent to a quotient of a classic root system of a finite Weyl group.
\end{Theorem}

As already mentioned above, for irreducible \GRSs of rank 2, the assertion was proved in \cite{DF23} (Theorem 5.5). Our approach for treating the higher rank case is based on the observation that one can associate a \emph{reduced root set} $\WG(R)$ with
any \GRS $(R,V)$. As mentioned in the introduction of \cite{DF23}, the reduced root set of a \GRS provides a crystallographic arrangement in the sense of \cite{p-C10}.
Using the classification of finite Weyl groupoids in  \cite{CH15}, the first author
obtained a complete classification of crystallographic arrangements in \cite{p-C10}. The proof of our main result relies strongly on these classification results.

Let $(R,V)$ be a \GRS. Then its reduced root set $\WG(R)$ is the set of its \emph{primitive} elements and it forms a root set of a Weyl groupoid. The only additional information encoded in $(R,V)$ are possible \emph{multipliers} of roots, i.e., for a primitive $\alpha\in R$, the maximal $k\in\mathbb{N}$ such that $\alpha,2\alpha,\ldots,k\alpha\in R$ is called the multiplier of $\alpha$.

In rank two, there are infinitely many finite Weyl groupoids and only finitely many of them are \redGRSs of a \GRS. In rank three, only 15 of the 55 finite Weyl groupoids are related to \GRSs.
Thus one task in this note is to determine those Weyl groupoids which are not induced by a \GRS.
The other task is to check that all the remaining ones are restrictions of Weyl arrangements, and to determine the possible multipliers of all roots.

As a consequence of our main result, the main result of \cite{p-CL-15} implies that every \redGRS of a \GRS is the root system of a Nichols algebra.
This is further evidence that \GRSs are natural objects.

\section{Properties of generalized root systems}

\begin{Definition}[{\cite[Def.\ 1.1]{DF23}}]\label{def:GRS}
Let $(V,(\cdot,\cdot))$ be a finite dimensional euclidean vector space and $\emptyset\ne R\subseteq V$ a finite subset.
The pair $(R,V)$ is called a \emph{generalized root system} if
\begin{enumerate}
\item[(i)] $V = \langle R \rangle$,
\item[(ii)] for all $\alpha,\beta\in R$:
\begin{eqnarray*}
(\alpha,\beta)<0 & \Longrightarrow & \alpha+\beta\in R, \\
(\alpha,\beta)>0 & \Longrightarrow & \alpha-\beta\in R, \\
(\alpha,\beta)=0 & \Longrightarrow & (\alpha+\beta\in R \Longleftrightarrow \alpha-\beta\in R).
\end{eqnarray*}
\end{enumerate}
The elements of $R$ are called \emph{roots}, the \emph{rank} of $(R,V)$ is the dimension of $V$.
\end{Definition}

\begin{Remark}\label{scalar0}
We will frequently use the following property of roots in a \GRS:\\ If $\alpha$, $\beta$ are roots such that $\alpha+\beta\notin R$ and $\alpha-\beta\notin R$, then $(\alpha,\beta)=0$.
\end{Remark}

\begin{Lemma} \label{basiclemma}
Let $(R,V)$ be a \GRS. Then the following hold.

\begin{enumerate}
    \item $R = -R$.
    \item if $0 \neq \alpha \in R$, $\mu := \min \{ \lambda \in \mathbb{R}_{>0} \mid \lambda \alpha \in R \}$ and
    $\beta := \mu \alpha$, then there exists $0 \neq k \in \mathbb{N}$ such that
    $$\{ \lambda \in \mathbb{R} \mid \lambda \beta \in R \} = \{ i \in \mathbb{Z} \mid -k \leq i \leq k \}.$$
    \item If $R$ is of rank $1$, then there exists $0 \neq \alpha \in R$ and $0 \neq k \in \mathbb{N}$
    such that $R = \{ j \alpha \mid j \in \mathbb{Z}, -k \leq j \leq k \}$.
\end{enumerate}
\end{Lemma}

\begin{proof} Let $0 \neq \alpha \in R$. Then $(\alpha,\alpha) >0$ and therefore $0 = \alpha -\alpha \in R$
by  \ref{def:GRS}. As $0,\alpha \in R$, $(0,\alpha) = 0$ and $\alpha = 0+\alpha$, it follows
from \ref{def:GRS} that $-\alpha = 0 -\alpha \in R$. This proves the first assertion.

As $R$ is finite, the set $\{ \lambda \in \mathbb{R}_{>0} \mid \lambda \alpha \in R \}$ is finite
and therefore $\mu$ exists. 

Let $\beta := \mu \alpha$ and  $\Lambda := \{ \lambda \in \mathbb{R}_{>0} \mid \lambda \beta \in R \}$. Then $\lambda \geq 1$
for all $\lambda \in \Lambda$. 
Suppose $1 < \lambda \in \Lambda$. As  $(\lambda \beta,\beta)>0$
we have $(\lambda - 1)\beta \in R$ by \ref{def:GRS} and hence $\lambda-1 \in \Lambda$.
It follows that  $\Lambda = \{1,2,\ldots,k\}$
for some $0 \neq k \in \mathbb{N}$. Now, the second assertion of the lemma follows from the first
and the third is a consequence of the second.
\end{proof}

\begin{Definition}[{\cite[Def.\ 1.1]{DF23}}]
Let $(R,V)$ be a \GRS.
A root $0\ne\alpha\in R$ is called \emph{primitive} if $\alpha=k\alpha'\in R$, $\alpha'\in R$, and $k\in\mathbb{Z}_{>0}$ implies $k=1$.
If $\alpha\in R$ is primitive and $\alpha,2\alpha,\ldots,k\alpha\in R$, $(k+1)\alpha\notin R$, then $k$ is called the \emph{multiplier} of $\alpha$.
\end{Definition}

\begin{Remark} Let $(R,V)$ be a \GRS and $\alpha \in R$. It follows from Lemma \ref{basiclemma}
that $\alpha$ is primitive if and only if $-\alpha$ is primitive in which case
they have the same multiplier. Moreover, if $\alpha$ is primitive with multiplier $k$,
then $\mathbb{R} \alpha \cap R = \{ j \alpha \mid j \in \mathbb{Z}, -k \leq j \leq k \}$.
\end{Remark}

\begin{Definition}[{\cite[Def.\ 1.4]{DF23}}]
Let $(R,V)$ be a \GRS.
A basis $S\subseteq R$ of $V$ is called a \emph{base} if every element of $R$ is a non-negative or non-positive integral linear combination of $S$.
\end{Definition}

\begin{Lemma}[{\cite[Cor.\ 1.7]{DF23}}]
Let $(R,V)$ be a \GRS and let $\alpha \in R$ be primitive.
Then there exists a base $S$ such that $\alpha \in S$.
\end{Lemma}

\begin{Definition}[{\cite[Def.\ 1.1]{DF23}}]
Let $(R,V)$ be a \GRS of rank $r$, $S$ a base of $(R,V)$, and let $\gamma_S : V \rightarrow \mathbb{Z}^r$ be the coordinate map that maps a root to its coordinate vector with respect to $S$.
We call
$$ \WG(R) := \{ \alpha \in R \mid 0\ne \alpha \text{ is primitive} \} $$
the \emph{\redGRS} of $(R,V)$ and denote
$$ \WG_S(R) := \{ \gamma_S(\alpha) \in R \mid 0\ne \alpha \text{ is primitive} \}. $$
\end{Definition}

\begin{Definition}[{\cite[Section 1]{DF23}}]
Two \GRSs $(R_1,V_1)$, $(R_2,V_2)$ are called \emph{equivalent} if there is a vector space isomorphism $V_1\rightarrow V_2$ that maps $R_1$ to $R_2$.
\end{Definition}

\begin{Definition} Let $(R,V)$ be a \GRS. A subset $R'$ of $R$ is called a \emph{parabolic}
subset of $(R,V)$
if $R \cap \langle R' \rangle = R'$.
\end{Definition}

\begin{Lemma}[{\cite[Def.\ 1.1]{DF23}}] Let $(R,V)$ be a \GRS,  $R' \subseteq R$  a parabolic subset of $(R,V)$
and $V' := \langle R' \rangle$. Then $(R',V')$ is a \GRS with respect to the scalar product 
induced from $V$ onto $V'$.
Moreover, $\WG(R') = \WG(R) \cap R'$
\end{Lemma}

\begin{proof} Both assertions are straightforward
from the definitions.
\end{proof}

\begin{Lemma} \label{parabolic3crit}
Let $(R,V)$ be a \GRS and assume that $0\ne \alpha\in R$ is contained in a parabolic subsystem $P$ of rank $2$ 
such that $|\WG(P)|=6$. Then the multiplier of $\alpha$ is $1$.
\end{Lemma}
\begin{proof}
Let $k$ be the multiplier of $\alpha$. Then $k\alpha$ is contained in the parabolic subsystem. By \cite[Theorem 5.2]{DF23}, the only \redGRS of rank $2$ with $6$ elements is the one labeled 1(i) in \cite[Section 5.2]{DF23}. In this \GRS, no root has a multiplier greater than $1$, thus $k=1$.
\end{proof}

\begin{Lemma} \label{parabolic4crit} 
Let $(R,V)$ be a \GRS and assume that $0\ne \alpha\in R$ is a primitive root of 
$(R,V)$ such that $2 \alpha \in R$. Assume that $\alpha$ is contained in a parabolic subsystem $P$ of rank 2
such that $|\WG(P)|=8$. 
Then there exists a primitive root $\gamma \in \WG(P)$
such that $\WG(P) = \Pi \cup -\Pi$ where $\Pi = \{ \alpha, \gamma, \gamma - \alpha, \gamma+\alpha \}$.
\end{Lemma}

\begin{proof} Let $P$ be a parabolic root system containing
$\alpha$ and hence also $2\alpha$ such that $|\WG(P)|=8$. By  \cite[Theorem 5.2]{DF23}, a \GRS  of rank 2 such that the \redGRS contains $8$ elements is equivalent to one of the \GRS 1(ii), 2(i) or 2(ii) listed in 
\cite[Section 5.2]{DF23}. Thus, there is a linear isomorphism $\varphi$ from 
$\langle P \rangle$ to $\langle x,y \rangle$ (where $x$ and $y$
are as in \cite[5.2.\ 1(ii), 2(i) or 2(ii)]{DF23})
such that $\varphi(P)$ is mapped onto the set of roots $R'$
of one those \GRS. As $\alpha$ and $2\alpha$ are in $P$, the 
case 1(ii) is impossible. In Case 2(i), $\varphi(\alpha) \in 
\{ x,-x \}$
and in Case 2(ii), $\varphi(\alpha) \in \{ x,-x,y,y \}$.
If $\varphi(\alpha) \in \{ x,-x \}$, then the claim holds
with $\gamma := \varphi^{-1}(y)$ and if 
$\varphi(\alpha) \in \{ y,-y \}$ then the claim holds
with $\gamma := \varphi^{-1}(x)$.
\end{proof}

\section{Crystallographic arrangements}

We briefly recall the notions of simplicial and crystallographic arrangements (cf.\ \cite[1.2, 5.1]{OT}, \cite{p-C10}, \cite{C19}).

\begin{Definition}\label{A_R_2}
Let $r\in\mathbb{N}$, $V:=\mathbb{R}^r$, and $\mathcal{A}$ be a finite set of linear hyperplanes in $V$, i.e.\ an \df{arrangement of hyperplanes}.
Let $\mathcal{K}(\mathcal{A})$ be the set of connected components (\df{chambers}) of $V\backslash \bigcup_{H\in\mathcal{A}} H$.
If every chamber $K$ is an \df{open simplicial cone}, i.e.\ there exist
$\alpha^\vee_1,\ldots,\alpha^\vee_r \in V$ such that
\begin{equation*}
K = \Big\{ \sum_{i=1}^r a_i\alpha^\vee_i \mid a_i> 0 \quad\forall\:\:
i=1,\ldots,r \Big\} =: \langle\alpha^\vee_1,\ldots,\alpha^\vee_r\rangle_{>0},
\end{equation*}
then $\mathcal{A}$ is called a \df{simplicial arrangement}.
\end{Definition}

\begin{Definition}\label{def:quotient}
Let $(\mathcal{A},V)$ be an arrangement of hyperplanes and $X\le V$.
The \emph{restriction} $\mathcal{A}^X$ of $\mathcal{A}$ to $X$ is defined by 
$$ \mathcal{A}^X := \{ X\cap H \mid H \in \mathcal{A},  X\not\subseteq H \} $$
and is an arrangement in $X$.

Let $(R,V)$ be a \GRS, $\mathcal{A}:=\{\alpha^\perp \mid \alpha\in R\}$, and $Y\le V$ be a subspace generated by elements of $R$. We write $X:=Y^\perp$.
Then the projection $V\rightarrow X$ defines a \GRS $(R',X)$ for the restriction $\mathcal{A}^X$ which is called a \emph{quotient} in \cite{DF23}.

Thus quotients are restrictions of arrangements of hyperplanes to elements of the intersection lattice of the arrangement.
\end{Definition}

\begin{Definition}[{\cite[Definition 2.3]{p-C10}}]\label{cryst:arr}
Let $\mathcal{A}$ be a simplicial arrangement in $V$ and $\RS\subseteq V^*$ a finite set
such that $\mathcal{A} = \{ \ker \alpha \mid \alpha \in \RS\}$ and $\mathbb{R}\alpha\cap \RS=\{\pm \alpha\}$
for all $\alpha \in \RS$.
We call $(\mathcal{A},V,\RS)$ a \df{crystallographic arrangement} if
for all chambers $K\in\mathcal{K}(\mathcal{A})$:
\begin{equation}\label{equ:crys}
\RS \subseteq \sum_{\alpha \in B^K} \mathbb{Z} \alpha,
\end{equation}
where
\[ B^K = \{ \alpha\in \RS \mid \forall x\in K\::\:\alpha(x)\ge 0,\:\: \langle \ker \alpha \cap \overline{K}\rangle = \ker \alpha \} \]
corresponds to the set of walls of $K$.
\\
Two crystallographic arrangements $(\mathcal{A},V,\RS)$, $(\mathcal{A}',V,\RS')$ in $V$ are called \df{equivalent}
if there exists $\psi\in\Aut(V^*)$ with $\psi(\RS)=\RS'$; we write $(\mathcal{A},V,\RS)\cong(\mathcal{A}',V,\RS')$.
\\
If $\mathcal{A}$ is an arrangement in $V$ for which a set $\RS\subseteq V^*$ exists such that $(\mathcal{A},V,\RS)$ is crystallographic, then we say that $\mathcal{A}$ is \df{crystallographic}.
\end{Definition}

\noindent
We use the convenient notation introduced in \cite{C19}:

\begin{Definition}{{\cite[Def.\ 3.3]{C19}}}\label{rootscartan}
Let $(\mathcal{A},V,\RS)$ be a crystallographic arrangement and $K$ a chamber.
Fixing an ordering for $B^K$, we obtain a coordinate map
$$ \coord^K : V\rightarrow \mathbb{R}^r \quad \text{with respect to}\quad B^K.$$
The elements of the standard basis $\{\alpha_1,\ldots,\alpha_r\}=\coord^K(B^K)$ are called \df{simple roots}.
The set
\[ R^K := \{ \coord^K(\alpha) \mid \alpha\in \RS \} \subseteq \mathbb{N}_0^r \cup -\mathbb{N}_0^r \]
is called the set of \df{roots} of $\mathcal{A}$ at $K$. The roots in $R^K_+:=R^K\cap \mathbb{N}_0^r$ are called \df{positive}.
\end{Definition}

We now identify $V$ and $V^*$ via the euclidean form.
The proof of our main theorem relies on the following observation which
is stated in \cite[Proposition 1.6]{DF23} and proven for example in \cite[\S 10.1]{Hu72} for classic root systems:

\begin{Proposition} \label{GRStoWeylgroupoid}
Let $(R,V)$ be a \GRS and $\mathcal{A}:=\{\alpha^\perp \mid \alpha\in R\}$.
Then $(\mathcal{A},V,\WG(R))$ is a crystallographic arrangement.
Moreover, if $(R,V)$ is irreducible, then $(\mathcal{A},V,\WG(R))$
is irreducible.
\end{Proposition}

\begin{proof}
For each chamber $K$, choose $\zeta\in V$ in the interior of $K$ and set $\Pi:=\{\alpha \in R \mid (\alpha,\zeta)>0 \}$. Note that $R=\Pi\cup -\Pi \cup \{0\}$ because for $\alpha\in R$, $(\alpha,\zeta)=0$ implies $\alpha=0$ since $\zeta$ is in the interior of $K$. Let $B^K$ be the set consisting of those elements of $\Pi$ which do not decompose as a sum of elements of $\Pi$.
If $\beta\in \Pi$, then after decomposing $\beta$ into sums of elements of $\Pi$ we obtain
$\beta\in \sum_{\alpha\in B^K} \mathbb{Z}_{\ge 0} \alpha$; the same argument works for $\beta\in -\Pi$, and of course $0$ is also an integral combination of $B^K$.\\
It remains to check that $B^K$ corresponds to the set of walls of $K$ and that $K$ is an open simplicial cone.
This is literally the proof of \cite[\S 10.1, Theorem' (3),(4),(5)]{Hu72}.
\end{proof}

With the above notation, $B^K=S$ is a base of $(R,V)$ and we have $\WG_S(R) = R^K$.
Hence for any \GRS $(R,V)$ and each base $S$ we obtain that $\WG_S(R)$ is a root set of a crystallographic arrangement (or equivalently of a Weyl groupoid).

We can now consider each crystallographic arrangement and determine the corresponding \GRSs. The following is a complete list of cases:

\begin{Theorem}[cf.\ {\cite{CH15}, \cite{p-C10}}]\label{cryarrclas}
There are (up to equivalence) exactly three families of irreducible crystallographic arrangements 
of rank at least $2$:
\begin{enumerate}
\item The family of rank two parametrized by triangulations of convex $n$-gons by non-intersecting diagonals.
\item For each rank $r>2$, arrangements of type $A_r$, $B_r$, $C_r$
and $D_r$, and a further series of $r-1$ arrangements.
\item Another $74$ ``sporadic'' arrangements of rank $r$, $3\le r \le 8$.
\end{enumerate}
\end{Theorem}
The $74$ sporadic root sets are listed in \cite[\S B]{CH15}. For a concrete description of the series of rank at least 3 see 
Remark \ref{remseries}.

\subsection{Proof of Theorem \ref{mainresult}}\label{proof:main}
Let $(V,R)$ be an irreducible \GRS of rank at least $2$.
By Proposition \ref{GRStoWeylgroupoid}, its reduced root set $\WG(R)$ is the set of roots of an irreducible
crystallographic arrangement of rank at least $2$.
By Theorem \ref{cryarrclas} the irreducible crystallographic arrangements are classified and they are subdivided into three classes.
If $\WG(R)$ is of rank $2$, then the assertion follows from Theorem \ref{proprank2}.
If $\WG(R)$ is a member of series of rank at least $3$, the assertion follows from Proposition \ref{propseries}.
Finally, if $\WG(R)$ is one of the $74$ sporadic crystallographic arrangements, then the assertion follows from Proposition \ref{propsporadic}.

\section{Infinite series}

\subsection{Rank two}
The \GRSs of rank two are classified in \cite[Theorem 5.2]{DF23} and have already been identified as quotient root systems. For convenience we recall Theorem 5.5 of \cite{DF23}.

\begin{Theorem}[{\cite[Theorem 5.5]{DF23}}]\label{proprank2}
Every irreducible \GRS of rank 2 is equivalent to 
a quotient of a classic root system.
\end{Theorem}

\subsection{Series in rank greater than two}
Let $n \geq 3$ be a natural number and put $[n] :=
\{1,\ldots,n \}$. Let $V$ be an $n$-dimensional real vector space and
let $B = (b_1,\ldots,b_n)$ be a basis of $V$. 
 We put
\begin{eqnarray*}
A_{n-1} &:=& \{ b_i - b_j \mid 1 \leq i \neq j \leq n \}, \\
D_n  &:=& \{ \varepsilon b_i+ \varepsilon' b_j \mid 1 \leq i \neq j  \leq n,\:\: \varepsilon,\varepsilon' \in \{1,-1\}\}, \\
B_n  &:=& D_n \cup \{ \varepsilon b_i \mid 1 \leq i \leq n,\:\: \varepsilon \in \{ 1,-1\} \}
\end{eqnarray*}
and for $J \subseteq [n]$ we put
\begin{eqnarray*}
X_J &:=& \{ 2 \varepsilon b_j \mid j \in J,\:\: \varepsilon \in \{ 1,-1\} \}, \\
DC_n^J &:=& D_n \cup X_J \quad\mbox{and}\quad BC_n^J := B_n \cup X_J.
\end{eqnarray*}

We put $C_n := DC_n^{[n]}$. Note that $B_n = BC_n^{\emptyset}$ and $D_n = DC_n^{\emptyset}$.

\begin{Remark}\label{remseries}
The sets $A_{n-1}$, $B_n$, $DC_n^J$ are the root sets of the crystallographic arrangements in the series mentioned in Theorem \ref{cryarrclas} $(2)$.
\end{Remark}

\begin{Lemma} \label{lemseries1}
Let $(\cdot,\cdot)$ be a scalar product on $V$ such
that $B$ is an orthonormal basis of $V$. Then the following
hold.

\begin{enumerate}

\item Let $R \in \{ A_{n-1}, B_n, C_n, D_n \}$. Then
$(V,R \cup \{ 0_V \})$ is a \GRS that is
a (trivial) quotient of a classic root system and
$\WG(R) =R$.
\item Let $J \subseteq [n]$ and $R = DC_n^J$.
Then $(V,R \cup \{ 0_V \})$ is a \GRS that is equivalent to
a quotient of a classic root system and $\WG(R) =R$.
\item Let $J \subseteq [n]$ and $R = BC_n^J$.
Then $(V,R \cup \{ 0_V \})$ is a \GRS that is equivalent to
a quotient of a classic root system and $\WG(R) =B_n$.
\end{enumerate}
\end{Lemma}

\begin{proof} It is straightforward to check
that $R$ is a \GRS and to determine $\WG(R)$
in all cases. The only assertions that are less
straightforward are the claims that the \GRS
$DC_n^J$ and $BC_n^J$ are equivalent to quotients
of classic root systems. 

The quotients of the classic root systems of type 
$A_l,B_l,C_l$ and $D_l$ are described in Section 6.1
of \cite{DF23}. Using this one can verify that
the \GRS $DC_n^J$ is equivalent to a quotient of 
the classic root system $D_n$ and that the \GRS $BC_n^J$
is equivalent to a quotient of the classic root system
$B_n$. 
\end{proof}

\begin{Lemma} \label{serieslemma2}
Let $(R,V)$ be a \GRS of rank greater than two and assume that its \redGRS $\WG(R)$ is
$A_{n-1}$, $B_n$ or $DC_n^J$ for some $J$. Assume that $\alpha \in \WG(R)$ is such that $2\alpha \in R$.
Then $\WG(R)$ is the \redGRS $B_n$ and $\alpha = \varepsilon b_i$ for some $\varepsilon \in
\{ 1,-1 \}$ and some $1 \leq i \leq n$.

In particular, if the \redGRS $\WG(R)$ is not $B_n$, then $\WG(R) = R \setminus \{ 0 \}$.

\end{Lemma}
\begin{proof} Let $\alpha \in \WG(R)$ be such that $2\alpha \in R$.

Suppose, by contradiction, that there are $1 \leq i \neq j \leq n$ such that $\alpha = b_i-b_j$ or $\alpha = b_i+b_j$.
Choose $1 \leq k \leq n$ such that $i \neq k \neq j$. Then
$\beta := b_i-b_k \in \WG(R) \subseteq R$. Let 
$P$ be the rank 2 parabolic of $R$
that contains
$\alpha$ and $\beta$. Then $|\WG(P)| = 6$. 
By Lemma \ref{parabolic3crit}
it follows that the multiplier of $\alpha$ in $R$ is 1
yielding a contradiction.

Thus, 
$\alpha = \lambda b_i$ for some $1 \leq i \leq n$
and some $\lambda \in \{ 1,-1,2,-2 \}$.

Suppose, again by contradiction, 
that $\alpha = 2b_i$ for some $1 \leq i \leq n$. 
Choose $1 \leq j \leq n$ with $j \neq i$.
Since $2b_i \in \WG(R)$, it follows $\WG(R) = DC_n^J$ for some $J \subseteq [n]$ with $i \in J$.
Therefore $\beta := b_i+b_j \in \WG(R)$. Let $P :=
R \cap \langle \alpha,\beta \rangle$ be the rank 2 parabolic
of $R$ containing $\alpha$ and $\beta$. 
If $2b_j$ is not in $P$,
then $\WG(P)$ contains $6$ roots and therefore the multiplier of $\alpha$
is $1$ by Lemma \ref{parabolic3crit}, yielding a 
contradiction. It follows that  $2b_j \in \WG(P)$ and that $\WG(P) = \Pi \cup -\Pi$
where $\Pi = \{ 2b_i,2b_j,b_i+b_j,b_i-b_j \}$. 
We conclude that there is no root $\gamma \in \WG(P)$
such that $\WG(P) = \Pi' \cup -\Pi'$ with
$\Pi' = \{ \alpha,\gamma, \gamma-\alpha,\alpha+\gamma \}$.
Now, Lemma
\ref{parabolic4crit} implies that $2 \alpha$ is not contained
in $R$ and we obtain a contradiction. 

As the multipliers of $\beta$ and $-\beta$
are equal for each root $\beta \in R$, it follows 
from the above that $\alpha \in \{ b_i,-b_i \}$
for some $1 \leq i \leq n$. As 
$$A_{n-1} \cap \{ b_i,-b_i \} = \emptyset
=DC_n^J \cap \{ b_i,-b_i \}$$ it follows that
$\WG(R) = B_n$.

Suppose that $\WG(R) \neq R \setminus \{ 0 \}$. Then
there exists a root $\beta \in \WG(R)$ such that $2 \beta 
\in R$. Thus, the last assertion is a consequence of
the above.
\end{proof}

\begin{Proposition} \label{propseries}
Let $(R,V)$ be a \GRS of rank greater than two. Then
the following hold.
\begin{enumerate}
    \item If $\WG(R) = A_{n-1}$, then $R = R \cup \{ 0\}$,

    \item if $\WG(R) = B_n$, then $R = BC_n^J \cup \{ 0 \}$ for some
    $J \subseteq [n]$,
    \item if $\WG(R) = DC_n^J$, then $R = R \cup \{ 0\}$.
\end{enumerate}

Moreover, in each case, the \GRS R is equivalent to  a quotient of a classic root system.
\end{Proposition}

\begin{proof}
If $\WG(R) = A_{n-1}$ or $\WG(R) = DC_n^J$ for some
$J \subseteq [n]$ it follows from the last assertion
of Lemma \ref{serieslemma2} that $R = \WG(R) \cup \{ 0 \}$.

Suppose that $\WG(R) = B_n$. Let $\alpha \in B_n$ be
of the form $\alpha =b_i+b_j$
or $\alpha = b_i-b_j$ for some $1 \leq i < j \leq n$.
Similar as in the proof
of Lemma \ref{serieslemma2} it follows that the multiplier
of $\alpha$ in R is $1$. Furthermore, using the classification
of irreducible \GRS of rank 2 (Section 5.2 in \cite{DF23}), it follows that the multiplier
of $b_i$ in $R$ is $1$ or $2$. It follows that 
$R = BC_n^J$ where $J$ is the set
of all $i \in [n]$ for which the multiplier of $b_i$ is $2$.

This proves the three first assertions and the 
last one follows from Lemma \ref{lemseries1}
\end{proof}

\section{Sporadic finite Weyl groupoids}

\begin{longtable}{l|c|c|c|l|l}
\cite{CH15} & $|\WG(R)|/2$ & $(|R|-1)/2$ & GRS & reference & \cite{DF23} \\
\hline
(3,1) & 10 & 11 & + & quotient GRS \ref{uqgrs_mult} & $\mathcal{E}_{6,3}^{III}$\\
(3,2) & 10 & 10 & + & quotient GRS \ref{uqgrs} & $\mathcal{E}_{6,3}^{II} \approx \mathcal{E}_{7,3}^{I}$\\
(3,3) & 11 & 12 & + & quotient GRS \ref{uqgrs_mult} & $\mathcal{E} _{7,3}^{II}$\\
(3,4) & 12 & - & - & isotropic elements \ref{isonomult} $(b)$ & $-$\\
(3,5) & 12 & - & - & isotropic elements \ref{isonomult} $(b)$ & $-$\\
(3,6) & 13 & 13 & + & special case, \ref{part_cases} & $\mathcal{E}_{7,3}^{III}$\\
(3,6) & 13 & 14 & + & special case, \ref{part_cases} & $\mathcal{E}_{8,3}^{I}$\\
(3,7) & 13 & 14 & + & quotient GRS \ref{uqgrs_mult} & $\mathcal{E}_{7,3}^{V}$\\
(3,8) & 13 & 16 & + & special case, \ref{part_cases} & $\mathcal{F}_{4,3}^{II} \cong \mathcal{E}_{7,3}^{VI}\cong \mathcal{E}_{8,3}^{VI}$\\
(3,9) & 13 & 13 & + & special case, \ref{part_cases} & $\mathcal{F}_{4,3}^{I}\cong \mathcal{E}_{7,3}^{IV}\cong \mathcal{E}_{8,3}^{VII}$\\
(3,10) & 14 & - & - & isotropic elements \ref{isonomult} $(a)$ & $-$\\
(3,11) & 15 & - & - & isotropic elements \ref{isonomult} $(a)$ & $-$\\
(3,12) & 16 & - & - & isotropic elements \ref{isonomult} $(a)$ & $-$\\
(3,13) & 16 & 17 & + & special case, \ref{part_cases} & $\mathcal{E}_{8,3}^{II}$\\
(3,14) & 17 & 20 & + & quotient GRS \ref{uqgrs_mult} & $\mathcal{E}_{8,3}^{IV}$\\
(3,15) & 17 & 18 & + & quotient GRS \ref{uqgrs_mult} & $\mathcal{E}_{8,3}^{III}$\\
(3,16) & 17 & - & - & isotropic elements \ref{isonomult} $(a)$ & $-$\\
(3,17) & 18 & - & - & isotropic elements \ref{isonomult} $(a)$ & $-$\\
(3,18) & 18 & - & - & isotropic elements \ref{isonomult} $(a)$ & $-$\\
(3,19) & 19 & - & - & isotropic elements \ref{isonomult} $(a)$ & $-$\\
(3,20) & 19 & 21 & + & special case, \ref{part_cases} & $\mathcal{E}_{8,3}^{VIII}$\\
(3,21) & 19 & - & - & isotropic elements \ref{isonomult} $(a)$ & $-$\\
(3,22) & 19 & - & - & isotropic elements \ref{isonomult} $(a)$ & $-$\\
(3,23) & 19 & 23 & + & quotient GRS \ref{uqgrs_mult} & $\mathcal{E}_{8,3}^{V}$\\
(3,24) & 20 & - & - & isotropic elements \ref{isonomult} $(a)$ & $-$\\
(3,25) & 20 & - & - & isotropic elements \ref{isonomult} $(a)$ & $-$\\
(3,26) & 20 & - & - & isotropic elements \ref{isonomult} $(a)$ & $-$\\
(3,27) & 21 & - & - & isotropic elements \ref{isonomult} $(a)$ & $-$\\
(3,28) & 21 & - & - & isotropic elements \ref{isonomult} $(a)$ & $-$\\
(3,29) & 21 & - & - & isotropic elements \ref{isonomult} $(a)$ & $-$\\
(3,30) & 22 & - & - & isotropic elements \ref{isonomult} $(a)$ & $-$\\
(3,31) & 25 & - & - & isotropic elements \ref{isonomult} $(a)$ & $-$\\
(3,32) & 25 & - & - & isotropic elements \ref{isonomult} $(a)$ & $-$\\
(3,33) & 25 & - & - & isotropic elements \ref{isonomult} $(a)$ & $-$\\
(3,34) & 25 & - & - & isotropic elements \ref{isonomult} $(a)$ & $-$\\
(3,35) & 26 & - & - & isotropic elements \ref{isonomult} $(a)$ & $-$\\
(3,36) & 26 & - & - & isotropic elements \ref{isonomult} $(a)$ & $-$\\
(3,37) & 27 & - & - & isotropic elements \ref{isonomult} $(a)$ & $-$\\
(3,38) & 27 & - & - & isotropic elements \ref{isonomult} $(a)$ & $-$\\
(3,39) & 27 & - & - & isotropic elements \ref{isonomult} $(a)$ & $-$\\
(3,40) & 28 & - & - & isotropic elements \ref{isonomult} $(a)$ & $-$\\
(3,41) & 28 & - & - & isotropic elements \ref{isonomult} $(a)$ & $-$\\
(3,42) & 28 & - & - & isotropic elements \ref{isonomult} $(a)$ & $-$\\
(3,43) & 29 & - & - & isotropic elements \ref{isonomult} $(a)$ & $-$\\
(3,44) & 29 & - & - & isotropic elements \ref{isonomult} $(a)$ & $-$\\
(3,45) & 29 & - & - & isotropic elements \ref{isonomult} $(a)$ & $-$\\
(3,46) & 30 & - & - & isotropic elements \ref{isonomult} $(a)$ & $-$\\
(3,47) & 31 & - & - & isotropic elements \ref{isonomult} $(a)$ & $-$\\
(3,48) & 31 & - & - & isotropic elements \ref{isonomult} $(a)$ & $-$\\
(3,49) & 34 & - & - & isotropic elements \ref{isonomult} $(a)$ & $-$\\
(3,50) & 37 & - & - & isotropic elements \ref{isonomult} $(a)$ & $-$\\
(4,1) & 15 & 15 & + & quotient GRS \ref{uqgrs} & $\mathcal{E}_{6,4}^{I}$\\
(4,2) & 17 & 17 & + & quotient GRS \ref{uqgrs} & $\mathcal{E}_{6,4}^{II}$\\
(4,3) & 18 & 18 & + & quotient GRS \ref{uqgrs} & $\mathcal{E}_{7,4}^{I}$\\
(4,4) & 21 & 21 & + & quotient GRS \ref{uqgrs} & $\mathcal{E}_{7,4}^{II}$\\
(4,5) & 22 & 23 & + & quotient GRS \ref{uqgrs_mult} & $\mathcal{E}_{7,4}^{III}$\\
(4,6) & 24 & 24 & + & quotient GRS \ref{uqgrs} & $\mathcal{F}_{4,4} \cong \mathcal{E}_{7,4}^{IV} \cong \mathcal{E}_{8,4}^{VI}$\\
(4,7) & 25 & 25 & + & quotient GRS \ref{uqgrs} & $\mathcal{E}_{8,4}^{I}$\\
(4,8) & 28 & 29 & + & quotient GRS \ref{uqgrs_mult} & $\mathcal{E}_{8,4}^{II}$\\
(4,9) & 30 & 30 & + & quotient GRS \ref{uqgrs} & $\mathcal{E}_{8,4}^{III}$\\
(4,10) & 32 & 33 & + & quotient GRS \ref{uqgrs_mult} & $\mathcal{E}_{8,4}^{IV}$\\
(4,11) & 32 & 36 & + & quotient GRS \ref{uqgrs_mult} & $\mathcal{E}_{8,4}^{V}$\\
(5,1) & 25 & 25 & + & quotient GRS \ref{uqgrs} & $\mathcal{E}_{6,5}$\\
(5,2) & 30 & 30 & + & quotient GRS \ref{uqgrs} & $\mathcal{E}_{7,5}^{I}$\\
(5,3) & 33 & 33 & + & quotient GRS \ref{uqgrs} & $\mathcal{E}_{7,5}^{II}$\\
(5,4) & 41 & 41 & + & quotient GRS \ref{uqgrs} & $\mathcal{E}_{8,5}^{I}$\\
(5,5) & 46 & 46 & + & quotient GRS \ref{uqgrs} & $\mathcal{E}_{8,5}^{II}$\\
(5,6) & 49 & 50 & + & quotient GRS \ref{uqgrs_mult} & $\mathcal{E}_{8,5}^{III}$\\
(6,1) & 36 & 36 & + & quotient GRS \ref{uqgrs} & $\mathcal{E}_{6,6}$\\
(6,2) & 46 & 46 & + & quotient GRS \ref{uqgrs} & $\mathcal{E}_{7,6}$\\
(6,3) & 63 & 63 & + & quotient GRS \ref{uqgrs} & $\mathcal{E}_{8,6}^{I}$\\
(6,4) & 68 & 68 & + & quotient GRS \ref{uqgrs} & $\mathcal{E}_{8,6}^{II}$\\
(7,1) & 63 & 63 & + & quotient GRS \ref{uqgrs} & $\mathcal{E}_{7,7}$\\
(7,2) & 91 & 91 & + & quotient GRS \ref{uqgrs} & $\mathcal{E}_{8,7}$\\
(8,1) & 120 & 120 & + & quotient GRS \ref{uqgrs} & $\mathcal{E}_{8,8}$\\
\caption{Overview for the sporadic finite Weyl groupoids\label{fig:overview}}
\end{longtable}

\subsection{Overview}
We use the notation, labels, and the root sets as listed in \cite{CH15}.
There are $74$ sporadic finite Weyl groupoids;
we write $(r,i)$ for the Weyl groupoid of rank $r$ with label $i$.

In Table \ref{fig:overview} we first give an overview of the different cases that can occur.
The entry ``isotropic elements'' means that the axioms of a \GRS would imply the existence of $\alpha\ne 0$ with $(\alpha,\alpha)=0$ on the elements of this \redGRS; in this case there is no corresponding \GRS.
Otherwise, there exist \GRSs. Note that
the only sporadic Weyl groupoid which does not uniquely determine a \GRS is $(3,6)$. This appears as a restriction of the root systems of types $E_7$ and $E_8$.

Detailed information about quotients of exceptional
classic root systems is provided in Subsection 6.3 and Table I of \cite{DF23}.
In the last column we reproduce some of this information
and identify the generalized root systems with the labels given in \cite{DF23}. Almost all of them are uniquely determined by the numbers of roots.
To distinguish $(3,6)$ and $(3,9)$ we use the fact that $(3,9)$ is a restriction of the Weyl arrangement of type $F_4$ (which is not the case for $(3,6)$).
Similarly, in contrast to $(3,6)$, $(3,7)$ is not a restriction of the arrangement of type $E_8$.

\subsection{Uniquely determined \GRS}\label{uqgrs}
For the Weyl groupoids with labels
$$
(3,2),(4,1),(4,2),(4,3),(4,4),(4,6),(4,7),(4,9),(5,1),(5,2),
$$
$$(5,3),(5,4),(5,5),(6,1),(6,2),(6,3),(6,4),(7,1),(7,2),(8,1),
$$
every root is contained in a parabolic subgroupoid of rank two with $6$ roots. Thus in a \GRS $(R,V)$ whose \redGRS is one of those, every root has multiplier $1$ by Lemma \ref{parabolic3crit}. Since the crystallographic arrangements of these root systems are restrictions of Weyl arrangements, the corresponding unique \GRSs are quotient root systems.
Note that these include the Weyl groups of types $F_4=(4,6)$, $E_6=(6,1)$, $E_7=(7,1)$, $E_8=(8,1)$.
\medskip

In the remaining cases, there are several situations that can occur. For each situation, we first explain an example and then list which cases can be treated in an analogous way.

\subsection{Weyl groupoids implying isotropic elements}\label{isonomult}\ \\
$(a)$
Consider as an example the Weyl groupoid of rank $3$ with label $10$. It has an object with positive roots:
$$(0,0,1),(0,1,0),(0,1,1),(0,1,2),(0,1,3),(1,0,0),(1,0,1),$$
$$(1,0,2),(1,1,1),(1,1,2),(1,1,3),(1,1,4),(1,2,3),(1,2,4).$$
Now consider the root $\alpha=(0,1,3)$. For $\beta=(1,0,2)$ we have $\alpha+\beta\notin R$ and $\alpha-\beta\notin R$.
Thus $(\alpha,\beta)=0$ by the third axiom of a \GRS (see Remark \ref{scalar0}). Similarly, $\alpha$ has to be orthogonal to $(1,1,2)$ and $(1,2,3)$. But then $\alpha$ is orthogonal to $\langle (1,0,2),(1,1,2),(1,2,3)\rangle = V$, and in particular $(\alpha,\alpha)=0$, a contradiction. Thus, there is no
\GRS such that its \redGRS is equivalent to this Weyl groupoid.\\
Table \ref{tab:isonomult} contains a list of all other Weyl groupoids which may be discarded with the same argument. The coordinates of the roots are those of the root sets displayed as representatives in \cite{CH15}.

\begin{table}
\centering
\begin{tabular}{c|c|c|l}
rank & label & $\alpha$ & in $\alpha^\perp$ by Rem.\ \ref{scalar0} \\
\hline
3 & 10 & (0,1,3) & (1,0,2),(1,1,2),(1,2,3) \\
3 & 11 & (1,0,0) & (0,1,0),(1,1,4),(1,2,3) \\
3 & 12 & (1,0,0) & (0,1,0),(1,1,4),(1,2,5) \\
3 & 16 & (0,1,0) & (1,0,0),(1,2,4),(2,3,5) \\
3 & 17 & (1,0,0) & (0,1,0),(1,1,4),(1,2,5) \\
3 & 18 & (0,1,3) & (1,0,1),(1,1,2),(1,2,2) \\
3 & 19 & (0,1,0) & (1,0,0),(1,1,2),(1,2,5) \\
3 & 21 & (0,1,3) & (0,2,1),(1,0,1),(1,1,2) \\
3 & 22 & (0,1,3) & (1,0,1),(1,1,2),(1,2,2) \\
3 & 24 & (0,1,0) & (0,2,5),(1,0,0),(1,1,2) \\
3 & 25 & (0,1,0) & (1,0,0),(1,1,2),(1,2,5) \\
3 & 26 & (0,1,3) & (1,0,2),(1,2,3),(1,3,4) \\
3 & 27 & (0,1,0) & (0,2,5),(1,0,0),(1,1,2) \\
3 & 28 & (1,0,0) & (0,1,0),(1,1,4),(1,2,5) \\
3 & 29 & (0,1,0) & (1,0,0),(1,3,6),(2,2,5) \\
3 & 30 & (0,1,3) & (1,0,2),(1,2,3),(1,3,4) \\
3 & 31 & (0,1,0) & (0,2,5),(1,0,0),(1,3,8) \\
3 & 32 & (0,1,0) & (1,0,0),(1,3,8),(2,3,7)
\end{tabular}
\quad
\begin{tabular}{c|c|c|l}
rank & label & $\alpha$ & in $\alpha^\perp$ by Rem.\ \ref{scalar0} \\
\hline
3 & 33 & (0,1,0) & (1,0,0),(1,3,5),(2,2,3) \\
3 & 34 & (0,1,3) & (1,0,2),(1,3,4),(2,2,5) \\
3 & 35 & (0,1,0) & (0,2,5),(1,0,0),(1,3,8) \\
3 & 36 & (0,1,0) & (1,0,0),(1,3,8),(2,3,7) \\
3 & 37 & (0,1,0) & (0,2,5),(1,0,0),(1,3,8) \\
3 & 38 & (0,1,0) & (1,0,0),(1,3,8),(2,2,5) \\
3 & 39 & (0,1,0) & (1,0,0),(1,3,8),(2,3,9) \\
3 & 40 & (0,1,0) & (0,2,5),(1,0,0),(1,3,8) \\
3 & 41 & (0,1,0) & (1,0,0),(1,3,8),(2,2,5) \\
3 & 42 & (0,1,0) & (1,0,0),(1,3,8),(2,2,5) \\
3 & 43 & (0,1,0) & (0,2,5),(1,0,0),(1,3,8) \\
3 & 44 & (0,1,0) & (0,2,5),(1,0,0),(1,3,8) \\
3 & 45 & (0,1,0) & (0,2,5),(1,0,0),(1,3,8) \\
3 & 46 & (0,1,0) & (0,2,5),(1,0,0),(1,3,8) \\
3 & 47 & (0,1,0) & (1,0,0),(1,1,3),(1,2,8) \\
3 & 48 & (0,1,0) & (0,2,5),(1,0,0),(1,3,8) \\
3 & 49 & (0,1,0) & (0,2,5),(1,0,0),(1,4,9) \\
3 & 50 & (0,1,0) & (0,2,5),(1,0,0),(2,3,7)
\end{tabular}
\bigskip
\caption{\ref{isonomult} (a), all cases.}
\label{tab:isonomult}
\end{table}

\medskip
\noindent
$(b)$
There are two more Weyl groupoids for which this argument applies, those of rank three with labels $4$ and $5$. However here we have to include information on the multipliers.
\medskip

\noindent
{\bf (3,4)} 
The Weyl groupoid of rank $3$ with label $4$ has an object with positive roots:
$$(0,0,1),(0,1,0),(0,1,1),(0,1,2),(0,1,3),(1,0,0),(1,0,1),(1,1,1),(1,1,2),(1,1,3),(1,2,3),(1,2,4).$$
Assume that this is the Weyl groupoid of a \GRS $R$.
In $R$, all roots except $(1,1,2)$ have multiplier $1$ because they are contained in a parabolic subsystem of rank $2$ with $6$ roots. The root $(1,1,2)$ has multiplier at most $2$ because it is contained in a parabolic subsystem of rank $2$ with $8$ roots.

Now consider the root $\alpha=(0,1,0)$. For $\beta=(0,1,2)$ we have $\alpha+\beta\notin R$ and $\alpha-\beta\notin R$.
Thus $(\alpha,\beta)=0$ by the third axiom of a \GRS. Similarly, $\alpha$ has to be orthogonal to $(0,1,3)$ and $(1,0,0)$. But then $\alpha$ is orthogonal to $\langle (0,1,2),(0,1,3),(1,0,0)\rangle = V$, and in particular $(\alpha,\alpha)=0$, a contradiction. Thus this Weyl groupoid has no corresponding \GRS.
\medskip

\noindent
{\bf (3,5)} The Weyl groupoid of rank $3$ with label $5$ has an object with positive roots:
$$(0,0,1),(0,1,0),(0,1,1),(0,1,2),(1,0,0),(1,0,1),(1,0,2),(1,1,1),(1,1,2),(1,1,3),(1,2,2),(1,2,3).$$
Assume that this is the Weyl groupoid of a \GRS $R$.
The argument in \ref{isonomult} $(a)$ does not work here, because the primitive roots do not produce a contradiction. However, we see that
no root has a multiplier greater than $1$ because all roots are contained in a parabolic subsystem of rank $2$ with $6$ roots.
Hence $R$ cannot be a \GRS because the root $(1,0,0)$ would be orthogonal to $\langle (0,1,0),(1,0,2),(1,1,3)\rangle = V$.

\subsection{Weyl groupoids with unique \GRS}\label{uqgrs_mult}
The Weyl groupoid of rank $3$ with label $1$ has an object with positive roots:
$$(0,0,1),(0,1,0),(0,1,1),(0,1,2),(1,0,0),(1,0,1),(1,0,2),(1,1,1),(1,1,2),(1,1,3).$$
Assume that this is the Weyl groupoid of a \GRS $R$.
In $R$, all roots except $(0,0,1)$ have multiplier $1$ because they are contained in a parabolic subsystem of rank $2$ with $6$ roots.
The root $(0,0,1)$ has multiplier at most $2$ because it is contained in a parabolic subsystem of rank $2$ with $8$ roots. However, $(0,0,1)$ requires multiplier 2 because otherwise $(0,1,0)$ would be orthogonal to $\langle (0,1,2),(1,0,0),(1,1,3)\rangle$ which is the whole space $V$.
Hence $R$ is uniquely determined and is a quotient \GRS.\\
A similar argument applies for the Weyl groupoids with labels
$$
(3,1),(3,3),(3,7),(3,14),(3,15),(3,23),(4,5),(4,8),(4,10),(4,11),(5,6);
$$
for these we obtain $1, 1, 1, 3, 1, 4, 1, 1, 1, 4, 1$ positive roots with multiplier $2$ respectively.

\subsection{Particular cases}\label{part_cases}
\medskip
{\bf (3,6)} The Weyl groupoid of rank $3$ with label $6$ has an object with positive roots:
$$(0,0,1),(0,1,0),(0,1,1),(0,1,2),(0,1,3),(0,2,3),(1,0,0),$$
$$(1,0,1),(1,1,1),(1,1,2),(1,1,3),(1,2,3),(1,2,4).$$
Assume that this is the Weyl groupoid of a \GRS $R$.
In $R$, all roots except $(1,1,2)$ have multiplier $1$ because they are contained in a parabolic subsystem of rank $2$ with $6$ roots. The root $(1,1,2)$ has multiplier at most $2$ because it is contained in a parabolic subsystem of rank $2$ with $8$ roots.
There are 2 possible choices of sets of multipliers for the roots.
If all multipliers are $1$, then
$$\begin{pmatrix}2 & 0 & -1/2 \\
0 & 3 & -3/2 \\
-1/2 & -3/2 & 1 
\end{pmatrix}$$
defines a bilinear form with respect to which $R$ is a \GRS.
If all multipliers are $1$ except for the root $(1,1,2)$ which has multiplier $2$, then
$$\begin{pmatrix}9/8 & 0 & -1/3 \\
0 & 2 & -1 \\
-1/3 & -1 & 2/3 
\end{pmatrix}$$
defines a bilinear form with respect to which $R$ is a \GRS.
Thus we obtain two different equivalence classes of \GRSs for this Weyl groupoid.

\medskip

\noindent
{\bf (3,8)} The Weyl groupoid of rank $3$ with label $8$ has an object with positive roots:
$$(0,0,1),(0,1,0),(0,1,1),(0,1,2),(1,0,0),(1,0,1),(1,0,2),$$
$$(1,1,1),(1,1,2),(1,1,3),(1,2,2),(1,2,3),(1,2,4).$$
Assume that this is the Weyl groupoid of a \GRS $R$.
In $R$, the roots $$(0,1,0),(0,1,2),(1,0,1),(1,1,1),(1,1,3),(1,2,3)$$ have multiplier $1$ because they are contained in a parabolic subsystem of rank $2$ with $6$ roots.
The roots $(0,0,1),(0,1,1),(1,0,0),(1,0,2),(1,1,2),(1,2,2),(1,2,4)$ have multiplier at most $2$ because they are contained in a parabolic subsystem of rank $2$ with $8$ roots.
The root $(0,0,1)$ requires multiplier 2 because otherwise $(1,0,0)$ would be orthogonal to $\langle (0,1,0),(1,0,2),(1,1,3)\rangle$ which is the whole space $V$.
The root $(0,1,1)$ requires multiplier 2 because otherwise $(1,0,0)$ would be orthogonal to $\langle (0,1,0),(1,1,3),(1,2,2)\rangle=V$.
The root $(1,1,2)$ requires multiplier 2 because otherwise $(1,0,0)$ would be orthogonal to $\langle (0,1,0),(1,1,3),(1,2,4)\rangle=V$.
There are 16 possible choices of sets of multipliers for the roots.
If the multipliers are $(2,1,2,1,1,1,1,1,2,1,1,1,1)$ (for the above ordering of the roots), then
$$\begin{pmatrix}3 & 0 & -1 \\
0 & 2 & -1 \\
-1 & -1 & 1 
\end{pmatrix}$$
defines a bilinear form with respect to which $R$ is a \GRS.
In all other cases, the axioms of a \GRS would produce non-trivial isotropic elements.
Thus we obtain one equivalence class of \GRSs for this Weyl groupoid.

\begin{Remark} By Proposition 6.3(iv) in \cite{DF23} one has
$$\mathcal{E}_{7,3}^{VI} \cong \mathcal{E}_{8,3}^{VI} \cong
\mathcal{F}_{4,3}^{II}$$ and the \GRS obtained above is
equivalent to these quotients.
\end{Remark}

\medskip

\noindent
{\bf (3,9)} The Weyl groupoid of rank $3$ with label $9$ has an object with positive roots:
$$(0,0,1),(0,1,0),(0,1,1),(0,1,2),(1,0,0),(1,0,1),(1,1,1),$$
$$(1,1,2),(2,0,1),(2,1,1),(2,1,2),(2,1,3),(2,2,3).$$
Assume that this is the Weyl groupoid of a \GRS $R$.
In $R$, the roots $$(0,1,0),(0,1,2),(2,0,1),(2,1,1),(2,1,3),(2,2,3)$$ have multiplier $1$ because they are contained in a parabolic subsystem of rank $2$ with $6$ roots.
In $R$, the roots $(0,0,1),(0,1,1),(1,0,0),(1,0,1),(1,1,1),(1,1,2),(2,1,2)$ have multiplier at most $2$ because they are contained in a parabolic subsystem of rank $2$ with $8$ roots.
Thus every root has multiplier at most $2$.
There are 128 possible choices of sets of multipliers for the roots.
If the multipliers are all equal to $1$, then
$$\begin{pmatrix}3/4 & 0 & -1/2 \\
0 & 2 & -1 \\
-1/2 & -1 & 1 
\end{pmatrix}$$
defines a bilinear form with respect to which $R$ is a \GRS.
In all other cases, the axioms of a \GRS would produce non-trivial isotropic elements.

\medskip

\noindent
{\bf (3,13)} The Weyl groupoid of rank $3$ with label $13$ has an object with positive roots:
$$(0,0,1),(0,1,0),(0,1,1),(0,1,2),(0,1,3),(1,0,0),(1,1,0),(1,1,1),$$
$$(1,1,2),(1,1,3),(1,2,1),(1,2,2),(1,2,3),(1,2,4),(1,3,4),(2,3,4).$$
Assume that this is the Weyl groupoid of a \GRS $R$.
In $R$, the roots $$(0,1,0),(0,1,1),(0,1,3),(1,0,0),(1,1,0),(1,1,1),(1,1,3),(1,2,1),(1,2,3),(1,2,4),(1,3,4),(2,3,4)$$ have multiplier $1$ because they are contained in a parabolic subsystem of rank $2$ with $6$ roots.
In $R$, the roots $(0,1,2),(1,1,2),(1,2,2)$ have multiplier at most $2$ because they are contained in a parabolic subsystem of rank $2$ with $8$ roots.
For the root $(0,0,1)$, we have to consider all the possible multipliers $1,\ldots,4$.
There are thus 32 possible choices of sets of multipliers for the roots.
If the multipliers are $2,1,1,1,1,1,1,1,1,1,1,1,1,1,1,1$, then
$$\begin{pmatrix}20/3 & -10/3 & 0 \\
-10/3 & 4 & -4/3 \\
0 & -4/3 & 1 
\end{pmatrix}$$
defines a bilinear form with respect to which $R$ is a \GRS.
In all other cases, the axioms of a \GRS would produce non-trivial isotropic elements.

\medskip

\noindent
{\bf (3,20)} The Weyl groupoid of rank $3$ with label $20$ has an object with positive roots:
$$(0,0,1),(0,1,0),(0,1,1),(0,1,2),(0,1,3),(0,1,4),(1,0,0),(1,1,0),(1,1,1),$$
$$(1,1,2),(1,1,3),(1,1,4),(1,2,2),(1,2,3),(1,2,4),(1,2,5),(1,2,6),(1,3,6),(2,3,6).$$
Assume that this is the Weyl groupoid of a \GRS $R$.
Except $(0,0,1)$, all roots in $R$ have multiplier $1$ because they are contained in a parabolic subsystem of rank $2$ with $6$ roots.
There are 4 possible choices of multipliers for the root $(0,0,1)$.
If the multiplier is $3$, then
$$\begin{pmatrix}12 & -6 & 0 \\
-6 & 8 & -2 \\
0 & -2 & 1 
\end{pmatrix}$$
defines a bilinear form with respect to which $R$ is a \GRS.
In all other cases, the axioms of a \GRS would produce non-trivial isotropic elements.

\subsection{Conclusion}

The following proposition is a consequence of the results
in this section.

\begin{Proposition} \label{propsporadic}
Let $(R,V)$ be a \GRS of rank at least $3$ such that $\WG(R)$ is the 
set of roots of one the $74$ sporadic Weyl groupoids.
Then $R$ is equivalent to a quotient of a classic root system
of type $E_6,E_7,E_8$ or $F_4$.
\end{Proposition}

\providecommand{\bysame}{\leavevmode\hbox to3em{\hrulefill}\thinspace}
\providecommand{\href}[2]{#2}

\end{document}